\documentclass[oneside,english]{amsart}
\usepackage[utf8]{inputenc}
 
\usepackage[ruled,linesnumbered]{algorithm2e}

\usepackage{amssymb}
\usepackage{amsmath}
\usepackage{amsthm}
\usepackage{url}
\usepackage{mathrsfs}

\newtheorem{prop}{Proposition}
\numberwithin{prop}{section}

\newtheorem{lem}{Lemma}
\numberwithin{lem}{section}

\numberwithin{deff}{section}

\newtheorem{theor}{Theorem}
\numberwithin{theor}{section}

\newtheorem{cor}{Corollary}
\numberwithin{cor}{section}

\newtheorem{probl}{Problem}
\numberwithin{probl}{section}

\newcommand{\F}{\mathcal{F}}

\newcommand{\N}{\mathbb{N}}
 
\newcommand{\Z}{\mathbb{Z}}

\newcommand{\expect}{\mathbb{E}}

\DeclareMathOperator{\dom}{dom}
\DeclareMathOperator{\supp}{supp}
\DeclareMathOperator{\cont}{cont}

\DeclareMathOperator{\h}{h}
\DeclareMathOperator{\pr}{pr}
\DeclareMathOperator{\Fin}{Fin}

\DeclareMathOperator{\uac}{\overline{AC}}

\DeclareMathOperator{\uacf}{\overline{AC}_{\F}}
\DeclareMathOperator{\lacf}{\underline{AC}_{\F}}
\DeclareMathOperator{\dhf}{d_\F^{H}}
\DeclareMathOperator{\HD}{D^{H}}

\title{Kolmogorov complexity and entropy of amenable group actions}
\author[1]{Andrei Alpeev}
\address{Andrei Alpeev,
	Chebyshev Laboratory, St. Petersburg State University, 14th Line, 29b, Saint Petersburg, 199178 Russia}
\email{a.alpeev@spbu.ru}


\begin{document}



\begin{abstract}
It was proved by Brudno that entropy and Kolmogorov complexity for dynamical systems are tightly related. We generalize his results to the case of arbitrary computable amenable group actions. Namely, for an ergodic shift-action, the asymptotic Kolmogorov complexity of a typical point is equal to the Kolmogorov-Sinai entropy of the action. For topological shift actions, the asymptotic Komogorov complexity of every point is bounded from above by the topological entropy, and there is a point attaining this bound.
\end{abstract}
\maketitle
\tableofcontents
\section{Introduction}

The Kolmogorov complexity $C(x)$ of a string $x$ is, roughly speaking, the minimal length of a computer program that outputs exactly string $x$. The definition of Kolmogorov complexity formalizes the intuition that, while the decimal representation of the number ${77}^{{77}^{77}}$ is rather lengthy, it contains a very small amount of information. Kolmogorov complexity could be defined for arbitrary finitary objects. Namely, given a set of finitary objects together with an enumeration of this set by natural numbers, we define the Kolmogorov complexity of an element of this set as the Kolmogorov complexity of its index given by the enumeration.

Consider the space $A^\Z$, where $A$ is a finite set. We endow this space with the product topology, assuming the discrete topology on $A$.
The {\em shift-action} of group $\Z$ on the space $A^\Z$ is defined by the formula 
\[
(\mathbf{S}^i x)(j) = x (i+j),
\]
for $x \in A^\Z$ and $i,j \in G$. 
A {\em subshift} is a closed $\mathbf{S}$-invariant set $X \subset A^\Z$.
For a point  $x \in A^\Z$ we define its {\em upper asymptotic complexity} by the formula 
\[
\uac(x) = \limsup_{n \to \infty}\frac{C(\pr_{[-n,n]}(x))}{2n+1}.
\]
The {\em lower asymptotic complexity} is defined in the similar way, but with the lower limit.

In the works \cite{Br74}, \cite{Br82}, Brudno showed that the entropy of the aforementioned shift-action is intimately connected with the Kolmogorov complexity of its points. Namely, he proved that the asymptotic complexity of any point in the subshift is bounded from above by the topological entropy of the subshift, and that there is a point whose asymptotic complexity equals the entropy of the subshift. On the measure-theoretic side, he proved that, for an ergodic invariant measure on a shift-action, almost every point has the asymptotic complexity equal to the Kolmogorov-Sinai entropy of the measure-preserving action.

Later on, the topological part of these results was generalized to the case of a $\Z^d$ action by Simpson in \cite{Si15}. 
Alexander Shen brought my attention to the questions in the latter work concerning generalization of Brudno's results to the case of an arbitrary amenable group. 

In my 2013 diploma work \cite{A13} I've managed to generalize some of Simpson's result. Let us remind some necessary definitions.
Let $G$ be a countable amenable group. We remind that a countable group $G$ is called amenable if it contains a {\em F\o lner sequence}. A F\o lner sequence is a sequence $(F_i)$ of finite subsets such that 
\[
\lim_{i \to \infty} \frac{\lvert g F_i \setminus F_i \rvert}{\lvert F_i\rvert} = 0
\]
for every $g \in G$. 
A F\o lner sequence is called {\em tempered} if there is a constant $K>0$ such that the inequality 
\begin{equation*}
\left\lvert \bigcup_{j<i} F_j^{-1}F_i\right\rvert \leq  K \lvert F_i\rvert
\end{equation*}
holds for all $i \in \N$.
Note that from any F\o lner sequence, a tempered F\o lner subsequence could be refined by a simple iterative procedure, see \cite[Proposition 1.4]{L01}.

A countable group $G$ is called computable if its elements are one-to-one enumerated by natural numbers, and there is a computable function $f$ of two arguments such that for any two $g,h \in G$ the index of $gh$ is equal to $f$ applied to the indices of $g$ and $h$.

We will say that a F\o lner sequence $(F_i)$ is modest if $C(F_i) = o(\lvert F_i \rvert)$ (the definition of the Kolmogorov complexity extends to any constructible class, see Subsection \ref{subsec: computability and kolmogorov}). Every computable amenable group admits a modest F\o lner sequence, see Theorem \ref{thm: modest exists}. In addition,
Theorem \ref{thm: geometric criterion} from Section \ref{sec: modest} presents a very natural sufficient geometric criterion for a F\o lner sequence of a finitely-generated group to be modest. Namely, if each element of the F\o lner sequence is edge-connected on the Cayley graph and contains the group identity element, then this F\o lner sequence is modest. 

For a finite set $A$ consider the natural shift-action of $G$ on the space $A^G$. This action is defined by the formula 
\[
(gx)(h) = x(hg)
\]
for $x \in A^G$ and $g,h \in G$.
We endow $A^G$ with the product topology, assuming the discrete topology on $A$. A {\em subshift} in this context is any closed $G$-invariant set. For a subshift $X$ we denote $\h_G^{top} (X)$ the {\em topological entropy} of the action of $G$ restricted to $X$. For any subset $F$ of $G$ we denote $\pr_F$ the natural projection map from $A^G$ to $A^F$.

For any $x \in A^G$, we denote 
\[
\uacf(x) = \limsup_{i \to \infty} \frac{C(\pr_{F_i}(x))}{\lvert F_i\rvert}   
\]
the {\em upper asympotic complexity} of $x$ relative to $\F$, and
\[
\lacf(x) = \liminf_{i \to \infty} \frac{C(\pr_{F_i}(x))}{\lvert F_i\rvert}   
\]
the {\em lower asympotic complexity} of $x$ relative to $\F$.

The following theorem consitutes the main result of my master's thesis \cite{A13}. That work is written in Russian and has never been published, so I have decided to incorporate its results into this paper.

\begin{theor}\label{thm: topological}
Let $X \subset A^G$ be a subshift over a computable amenable group $G$. Let $\F$ be a modest F\o lner sequence. For every $x\in X$ holds 
\[
\uacf (x) \leq \h_G^{top} (X).
\]
If $\F$ is also tempered, then there is such a point $y \in X$ that 
\[
\uacf(y) = \lacf (y) = \h_G^{top} (X).
\]
If in addition $X$ has continuum cardinality, then there is a continuum of such points $y$.
\end{theor}
In Proposition \ref{prop: topological upper bound} we derive the upper bound for the asymptotic complexity in terms of the topological entropy. Proposition \ref{prop: topological entropy lower bound} states that there is a point with the complexity greater of equal than the topological entropy, and that the last assertion of the theorem above holds.

Afterward, Moriakov in the papers \cite{Mo15a} and \cite{Mo15b} independently considered the generalizations of Brudno's results to the case of amenable groups having so-called computable F\o lner monotiling. The latter means that the group has a special F\o lner sequence each of whose elemements can tile the whole group, and that this tiling can be generated by a computer program. This requirement seems to be rather restrictive, in particular, it is not known whether all amenable groups are monotileable. On the other hand, one can try to combine Moriakov's approach with the recent advances concerning tilings of amenable groups by finitely many F\o lner sets (see \cite{DHZ19} and also \cite{Ceta18}). I do not pursue that path in this paper.

In the current work I present a proof for the general case of measure-entropy Brudno's theorem: 

\begin{theor}\label{thm: measure}
Let $G$ be a computable amenable group. Let $\F$ be a modest tempered F\o lner sequence
Let $\mu$ be an ergodic invariant measure on $A^G$. Then for $\mu$-a.e. $x \in A^G$ we have 
\[
\uacf(x) = \lacf(x) = \h_G(A^G,\mu),
\]
where $\h_G(A^G,\mu)$ denotes the Kolmogorov-Sinai entropy of the action.
\end{theor}

Proof naturally splits into proving the corresponding lower and upper bounds. The lower bound for the complexity is established in the proposition \ref{prop: measure lower bound}. This bound is basically an adaptation of the argument from \cite{Si15}.

The upper bound constitutes Proposition \ref{prop: measure upper bound}. 
I consider this bound to be the main contribution of this paper. The proof has three main ingredients:
a Shannon-type frequency bound for asymptotic complexity; the monotonicity of asymptotic complexity under factor  maps between measure preserving systems;
the recent developements in generating partition theory due to Seward \cite{Se19}, Seward and Tucker-Drob \cite{SeTD}. The much weaker ``mean'' form of the bound was used by Bernshteyn in \cite{Be16}. 

The results of this paper have been announced in \cite{A18}.
 
It might be interesting to prove an ``effective''  analog of Theorem \ref{thm: measure}:
\begin{probl}
Does the statement of Theorem \ref{thm: measure} holds for a Martin-L{\"o}f random point?
\end{probl}

{\em Acknowledgements.} Research is supported by the Russian Science Foundation grant No14-21-00035. I would like to thank Alexander Shen for suggestion to work on generalizations of Brudno's results, and for helpful discussions. Discussions with Pavel Galashin were instrumental in proving Theorem \ref{thm: geometric criterion}. I'd like to thank Sergey Kryzhevich for his comments on my diploma work \cite{A13}; results from that work constitute a substantial part of this paper. I thank Pierre Guillon for the comments on the manuscript.   

\section{Preliminaries}

We will use the symbol ``$\Subset$'' for ``a finite subset''.

\subsection{Computability and Kolmogorov complexity}\label{subsec: computability and kolmogorov}
In the sequel we may assume without loss of generality that our notion of computability is augmented with a fixed oracle. This might change the notions relying on the notion of computablitity, but all the properties and arguments used in this paper remain true after the relativization. Of course, the notions used in the statements of the theorems should be adjusted accordingly.

Denote $2^*$ the set of all finite binary strings (the empty one included). 

For $x \in 2^*$ we denote $\lvert x \rvert$ the length of $x$. Let $f$ be a partial (i.e. not necessarily everywhere defined) computable function $2^* \to 2^*$. For $x \in 2^*$ we denote $C_f(x)$ the smallest $\lvert y\rvert$ such that $f(y) = x$ ($C_f(x) = +\infty$ if there are no such $y$'s). By the Kolmogorov-Solomonoff theorem (see \cite[Theorem 2.1.1]{LV97} or \cite[Theorem 1]{SUV}), there is an ``optimal'' computable function $f$, such that for any other computable $f'$ there is a constant $K$ for that inequality
\begin{equation}\label{eq: complexity optimality}
C_f(x) < C_{f'}(x) + K
\end{equation} 
holds for any $x \in 2^*$.
We fix any $f$ optimal in this sense. The {\em Kolmogorov complexity} of a string $x$ is defined as $C_f(x)$ and denoted by $C(x)$. It is easy to see that for any natural $k$ there are at most $2^k$ words $x \in 2^*$ satusfying the bound $C(x)<k$. An immediate consequence of the definition is that for any partial computable function $t$ there is a constant $K$ such that 
\begin{equation}\label{eq: first monotonicity}
C(t(x)) < C (x) + K
\end{equation}
for any $x \in 2^*$. Indeed, take $f' = t \circ f$ in the Inequality \ref{eq: complexity optimality}.

Enumeration is the usual device for extension of the notion of computability from the basic class of natural numbers (or binary strings). In the literature, for each new class of combinatorial objects (say, finite graphs or some colored finite structure) the enumeration is usually given in an {\em ad hoc} manner. While not causing much trouble in most cases, this approch seems inappropriate to me. Consider the following example. Let's say we want enumerate the class of finite graphs with vertices uniquely labeled by natural numbers from $1$ to the vertex count. One of generic ways is to consider string of the form: number of vertices + number of edges + list of edges (some sort of padding should be introduced in order to distiguish these parts), then we lexicogrpaphically enumerate all the ``correct'' discriptions, leaving for each graph only the lexicograpically first one. On the other hand, we can introduce the following kind of enumeration. Let $B$ be a non-enumberable and non-co-enumerable set of natural numbers. Consider such an enumertation of graphs that odd indices would correspond to graphs whose vertex number is in $B$, and even - whose vertex numbers are not in $B$. Intuitively, it seems like the first enumeration is ``natural'', while the second is not. In fact, this distinction would become very apparent if we would like to prove some computability statements for this class of labeled graphs. The distinction described is similar to that between G{\"o}del and non-G{\"o}del enumerations of computable functions. The purpose of the following is to sketch the way to formally separate  natural from non-natural enumerations for some classes of objects.

Let us give an inductive definition of a {\em constructible class} and a {\em natural enumeration} of a constructible class.
The set of natural numbers is a constructible class, and any computable bijection from it to any its decidable subset is a natural enumeration. Any finite set is a constructible class and any bijection from it to a subset of natural numbers is a natural enumeration. Let $U_0$ be a countable set together with a finite collection of constructible classes $U_1, \ldots , U_m$ and a finite collection $(f_i)$ of functions $f_i : U_{k^i_1} \times \ldots \times U_{k^i_{n_i}} \to U_0$ (called {\em construction functions}), with $k^i_j \in 0, \ldots, m$ and arities $n_i \in \N_0$ ($0$-ary functions are allowed, we will call them {\em constants}).
We will say that this forms a constructible class if there exists a bijection $q_0$ from $U_0$ to a decidable subset of natural numbers such that the following hold:
\begin{enumerate}
\item\label{iten: computability of natural functions} for some collection of natural enumerations $q_1, \ldots\, q_m$ for respective constructible classes, the functions $q_0 \circ f_i$ are computable (after identifying each $U_1, \ldots, U_m$ with a decidable subset of natuaral number via the respective enumeration);
\item the set $U_0$ is the minimal set closed under construction functions $(f_i)$, i.e. for any proper subset $U_0'$ of $U_0$ there is a construction function $f_i$ and arguments $u^j \in U_{k^i_j}$ such that $u^j \in U'_0$ for each $j$ with $k^i_j = 0$, and that $f_i(u^1, \ldots, u^{n_i}) \notin U'_0$.
\end{enumerate}

We say that any $q_0$ satisfying these requirements is a natural enumeration for this constructible class.

Note that the second clause means that any element of the constructible class could be obtained by an expression  that involve only contriction functions and elements from $U_1, \ldots, U_m$.

It is easy to prove inductively that all natural enumerations are computably equivalent:

\begin{prop}\label{prop: equivalence of enumerations}
Let $U$ be a constructible class and $q,q'$ be any two natural enumerations. Then the function $q' \circ q^{-1}$ is computable. Also, in the clause \ref{iten: computability of natural functions} of the definition above, one can substitute ``some collection of natural enumerations'' with ``any collection of natural enumerations''.
\end{prop}

\begin{proof}
The base of the induction (the constructible class of natural numbers and finite constructible classes) is trivial. 
Now let $U_0$ be a construstible class whose definition relies on constructible classes $U_1, \ldots, U_m$. By induction, we assume that the proposition holds for these.
Note that the last assertion of the proposition holds trivially for $U_0$ since the first assertion holds for $U_1, \ldots, U_m$. In order to verify the first assertion for $U_0$, it is enough to show that the graph of $q' \circ q^{-1}$ is a computably enumerable set. We fix a collection of natural enumeraitions $q_1, \ldots, q_m$ for $U_1, \ldots, U_m$ respectively. Since the second assertion of the proposition holds for $U_0$, compositions $q \circ f_i$ are computable functions if we identify $U_i$, $i = 1, \ldots, m$ with subsets of the natural numbers via $q_i$'s and $U_0$ via $q$. The same holds for $q'$ instead of $q$. So now we want to know which $q'$-indices correspond to $q$-indices. To do so we apply construction function in all possible ways, using only those values from $U_0$ where we know the correspondence of $q$ and $q'$ indices. In this way we will obtain new values of corresponding $q$ and $q'$ indices.
The fact that $U_0$ is generated by the construction functions implies that the procedure described will eventually give the corresponding $q$ and $q'$ indices for all elements of $U_0$.
\end{proof}

Instantly we have the following:

\begin{cor}
Let $U_0, \ldots, U_m$ be a finite collection of constructible classes.
If $f: U_1 \times \ldots \times U_m \to U_0$ is a function that is computable after identifying each $U_i$ with a subset of natural numbers via some natural enumeration, then this function will remain computable after identifying each $U_i$ with a subset of natural numbers via any natural enumeration.
\end{cor}

Let $U_0, \ldots, U_m$ be a collection of constructible classes. We will say that a function $f: U_1 \times \ldots \times U_m \to U_0$ is computable if it is computable after identifying each constructible class via some (=any) natural enumeration.

In practical situations it is hard to find unique natural collection of construction functions. The next corollary, that follows directly from the previous one, resolves this issue.

\begin{cor}
If two constructible classes share the same underlying set $U$, and the collection of construction function of the first one is a subset of the collection of construction function of the second one, then these two constructible classes have the same set of natural enumerations.
\end{cor}

Let us proceed with some examples. 

For any finite set $A$, the set $A^*$ of all finite strings is a constructible class. For the set of construction functions we can take a constant $f_0$ for the empty string, and a function $f_1: A^* \times A \to A^*$ which appends a symbol to the string. For a natural enumeration we can first order all the strings by their length and then lexicographically if the lengths ar equal. In the later examples we leave the construction of a natural enumeration as an exercise.

A product of any two constructible classes $U_1$ and $U_2$ could be naturally considered a constructible class. Indeed, fix a constant $f_0$ for an arbitrary element $w$ from $U_1 \times U_2$. Fix two additional construction functions $f_1: (U_1 \times U_2) \times U_1 \to U_1 \times U_2$ such that $f_1((u_1,u_2), u'_1) = (u'_1, u_2)$, and $f_2: (U_1 \times U_2) \times U_2 \to U_1 \times U_2$ such that $f_2((u_1,u_2), u'_2) = (u_1, u'_2)$. In fact, any finite product of constructible classes could be considered a constructible class in a natural fashion.

The set $\Fin(U)$ of all finite subsets of a contructible class $U$ could be naturally considered a constructible class. For a set of construction functions we take a constant $f_0$ returning the empty subset, and the function $f^1 : \Fin(U) \times U \to Fin(U)$ defined by $f_1(w, u) = w \cup \lbrace u\rbrace$.

For any constructible class $U$ and any finite set $A$, the set of partial maps with finite domains $A^{\Subset U}$ could be considered a constructible class. For the construction functions we take a constant $f_0$ returning the trivial partial map with the empty domain, and a function $f_1 : A^{\Subset U} \times U \times A \to A^{\Subset U}$ such that $(f_1(w, u, a)) (v) = w(v)$ if $v \neq u$ and $(f_1(w, u, a)) (v) = a$ otherwise (in other words, we either extend the domain and define the new function on $u$ to return $a$, or override the value of the old function on $u$ by $a$).


We identify $2^*$ with the set $\N$ by some natural enumeration. This way, we may define the Kolmogorov complexity as a function on the set of natural numbers. Note that this definition changes only up to an $O(1)$ summand if we change the natural enumeration, due to monotonicity of the Kolmogorov complexity under computable functions, Inequality \ref{eq: first monotonicity}, and computational equivalence of any two natural enumerations, Proposition \ref{prop: equivalence of enumerations}. In the same vein, for any constructible class $U$ we may fix a natural enumeration $q$ and define the Kolmogorov complexity for $u \in U$ as $C(q(u))$. Again, the complexity changes only up to an $O(1)$ summand if we change the natural enumeration.

One important example is complexity of tuples. Let $U_1, \ldots U_k$ be constructible classes. 
We introduce a special notation $C(u_1, \ldots, u_k) = C((u_1, \ldots, u_k))$.


Let us sum up some important properties of the Kolmogorov complexity.

\begin{prop}\label{prop: compl properties}
The following holds for the Kolmogorov complexity.
\begin{enumerate}
\item\label{compl prop: trivial bound} For any $x \in 2^*$ we have $C(x) \leq \lvert x\rvert + O(1)$. For any natural number $n$ we have $C(n) \leq \log n + O(1)$. For any finite set $A$ that have at least two elements, we have $C(w) \leq \lvert w\rvert \log \lvert A\rvert + O(1)$, for each $w \in A^*$.
\item\label{compl prop: combinatorics} For each constructible class $U$ and  $k \in \N$ there are no more than $2^k$ of such $x \in U$ that $C(x) < k$.
\item\label{compl prop: monot} Let $U_1, U_2$ be constructible classes and let $f: U_1 \to U_2$ be a computable function. There is a constant $K$ such that $C(f(x)) < C(x) + K$ for any $x \in U_1$.
\item\label{compl prop: tuple} Let $U_1,\ldots, U_m$ be a finite collection of constructible classes. There is a constant $K$ such that 
\begin{equation*}
\begin{aligned}
C(u_1, \ldots, u_k) \leq &C(u_1)+ \ldots + C(u_k) +\\ &2\log(C(u_1))+ \ldots + 2\log(C(u_{k-1})) + K,
\end{aligned}
\end{equation*}
for any $(u_1, \ldots, u_k) \in U_1 \times \ldots U_k$.
\item\label{compl prop: small set compl} For any constructible class $U$ there is a constant $K$ such that for any finite subset $T$ of $U$ and any $u \in T$ we have
\[
C(u) \leq C(T) + 2 \log (C(T)) + \log\lvert T\rvert + K.
\]
\item\label{compl prop: computable bound} Let $U$ be a constructible class, and $f$ be a computable function from $2^*$ to $U$. Then $C(f(x)) \leq \lvert x\rvert + O(1)$ for any $x \in 2^*$. If $t$ is a computable function function from $\N$ to $U$, then $C(t(i)) \leq \log i + O(1)$ for any $i \in \N$.
\end{enumerate}
\end{prop}

\begin{proof}

The first one easily follows from the definition of the Kolmogorov complexity via the comparison with the trivial function $f'(x) = x$. The variant for the natural numbers follows by considering the binary representation. The statement for the strings over finite alphabets follows easily by considering the lexicographical enumeration.

For the second one, consider the optimal function $f$ from the definition of the Kolmogorov complexity. Oviously, it assumes no more than $2^k$ values on the strings of length less than $k$. This implies the desired.

The third one is an easy consequence of Inequality \ref{eq: first monotonicity}.

We will prove the fourth one for binary strings. this would imply the statement in general. Let $U_1, \ldots, u_k$ be binary strings. Let $f: 2^* \to 2^*$ be the optimal function from the definition of the Kolmogorov complexity.
Let $y_1, \ldots, y_k$ be shortest preimages of respective $u_1, \ldots, u_k$. Let $t$ be a program that expects its input in the following form. First the binary representation of the length $\lvert u_1\rvert$ with each digit doubled, then the piece ``\verb|01|'', then string $y_1$, then the same goes for each $y_i$, $i = 2...k-1$, and in the end we just add string $y_k$. Program $t$ will read these, parse the strings $y_1, \ldots, y_k$, apply function $f$ to each, and apply a fixed natural enumeration function, thus recovering the index of the tuple $(u_1, \ldots, u_k)$. It is now easy to see that
\begin{equation*}
\begin{aligned}
C_t(u_1, \ldots, u_k) \leq &C(u_1)+ \ldots + C(u_k) +\\ &2\log(C(u_1))+ \ldots + 2\log(C(u_{k-1})) + 2(k-1).
\end{aligned}
\end{equation*}
Applying the definition of the Kolmogorov complexity, we obtain the desired inequality.

For the fifth one we again assume that our constructible class is the set binary strings. Let $t$ be a function from $\Fin(2^*) \times \N$  to $2^*$ such that $t(T, i)$ is the $i$-th in the lexicographical order element of $T \Subset 2^*$, if $i < \lvert T\rvert$, otherwise it is undefined. Note that for any finite $T$ and its element $u$ there is such $i < \lvert T\rvert$ that $t(T,i) = u$. Using previous properties, we observe that 
\[
C(u) = C(t(T,i)) + O(1) \leq C(T,i) + O(1) \leq C(T) + 2 \log (C(T)) + \log i  + O(1).
\]
This implies the desired.

The sixth property is a consequence of properties \ref{compl prop: trivial bound} and \ref{compl prop: monot}.

\end{proof}

We refer the reader to books \cite{LV97} and \cite{SUV} for additional details on the Kolmogorov complexity.


\subsection{Computable amenable groups}

Let $G$ be a countable group. A sequence $\F = (F_i)$ of finite subsets of group $G$, such that
\[
\lim_{i \to \infty}{ \frac{\lvert g F_i \setminus F_i \rvert}{\lvert F_i \rvert}} = 0
\]
holds for every $g \in G$, is called a {\em F\o lner sequence} for group $G$. A countable group is called an {\em amenable} group if it has a F\o lner sequence.

A F\o lner sequence is called {\em tempered} if there is a constant $K>0$ such that for all $i$ the following inequality is satisfied:
\begin{equation}
\left\lvert \bigcup_{j<i} F_j^{-1}F_i\right\rvert \leq  K \lvert F_i\rvert.
\end{equation}
We remind that from any F\o lner sequence, a tempered F\o lner subsequence could be refined (see \cite[Proposition 1.4]{L01}).

Let $G$ be a countable group. We will say that $G$ is a {\em computable group} if the set of its elements is the set of natural numbers, and the composition function $(g, h) \mapsto gh$ is computable. We may asssume also that $0$ is the identity element. We sometimes will use the order inherited from the set of natural numbers for computable groups. Note that a finitely-generated group is computable whenever the word problem for the group is decidable.
A {\em computable amenable group} is simply a computable group which is amenable. 
We can define the Kolmogorov complexity of a finite subset $F$ of $G$, since finite subsets form a constructible class (the class of finite subsets of natural numbers). We will say that a F\o lner sequence $(F_i)$ of a computable amenable group is {\em modest} if the following asymptotic bound holds:
\[
C(F_i) = o(\lvert F_i \rvert).
\]
This definition was given in my work \cite{A12}, where the  following result was also proved.
\begin{theor}\label{thm: modest exists}
Every computable amenable group has a modest F\o lner sequence.
\end{theor}
\begin{proof}[Sketch of proof]
Let us fix an enumeration of finite subset of $G$. For $i$ let $F_i$ be the first subset such that $\lvert F_i \rvert > i$ and that $\lvert g F_i \setminus F_i \lvert < \lvert F_i \rvert /(i+1)$ for every $g < i$ (the order is the order of natural numbers). It is easy to see that this is a F\o lner sequence. It also follows from Proposition \ref{prop: compl properties}, Property \ref{compl prop: computable bound} that $C(F_i) \leq \log i + O(1)$ (note that the map $i \mapsto F_i$ is computable). So $C(F_i) = o(\lvert F_i\rvert)$.
\end{proof}
In Section \ref{sec: modest} we will introduce a useful geometric criterion for a F\o lner sequence to be modest, which works in a lot of situations.

\subsection{Actions and ergodic theory}
We refer the reader to books \cite{Gl03} and \cite{EW11} for the introduction to ergodic theory. Results concerning the ergodic theory of amenable groups could be found in \cite{Mou85}, \cite{OW87}, \cite{L01}, \cite{KL16}. 

Let $G$ be a countable group. We will denote by $1_G$ the identity element of the group. Let $A$ be a finite alphabet. The shift-action of group $G$ on the space $A^G$ is defined by
\[
(gx)(h) = x(hg),
\] 
for every $x \in A^G$ and $g,h \in G$. We endow the space $A^G$ with the product topology. It is easy to note that the group $G$ acts by homeomorphisms in the shift action which we defined. A {\em subshift} is any closed invariant subset of $A^G$. 
We will also use term {\em subshift} for the action given by restricting the action of $G$ to the closed invariant subset $X$. For any $F \Subset G$ we denote by $\pr_F$ the natural projection map from $A^G$ to $A^F$.

Consider two actions of a group $G$ on sets $X$ and $Y$.
A map $\pi: X \to Y$ is said to be equivariant if $\pi(gx) = g \pi(x)$ for all $g \in G$ and $x \in X$.

We denote $A^{\Subset G}$ the set of all partial maps from $G$ to $A$ with finite domains. We denote $\supp (t)$ the domain of $t$ for $t \in A^{\Subset G}$. 
We denote by $\cont(t)$ the word 
\[
(t(g_1), \ldots, t(g_{\lvert \supp(t)\rvert})),
\]
from $A^{\lvert \supp(t)\rvert}$, where $\supp(t) = \lbrace g_1, \ldots, g_{\lvert \supp(t)\rvert}\rbrace$, and $g_i$'s are listed in the increasing order.

Let $G$ be a countable group. By an {\em action} of $G$ on a standard probability space $(X,\mu)$ we will always mean a measurable and measure-preserving action. 
A measure preserving action is {\em ergodic} if all invariant sets have measure $0$ or $1$. An invariant measure is {\em ergodic} if the corresponding action is ergodic.

Suppose a countable group $G$ acts on two standard probability spaces $(X,\mu)$ and $(Y,\nu)$. A measure-preserving map $\pi: Y \to X$ is said to be a {\em factor map} if $g(\pi(y)) = \pi(g y)$ for $\nu$-a.e. $y \in Y$. It is not hard to see that there is a full-measure $G$-invariant subset $Y'$ of $T$ such that the restriction if $\pi$ to $X'$ is an equivariant map.
A factor map is called an {\em isomorphism} if it is one-to-one on a set of full measure. 

Suppose a countable group $G$ acts on a standard probability space $(X,\mu)$. Let $\bar{\pi}: X \to A$ be a measurable map to a finite set $A$. We can extend this map to $\pi: X \to A^G$ by the equality
\[
(\pi(x))(g) = \bar{\pi}(gx).
\]
The resulting map is equivariant. In fact, any measurable equivariant map $X \to A^G$ could be obtained in such a way. We will call the map $\bar{\pi}$ the {\em generator map} for the equivariant map $\pi$.
Consider an equivariant map $\pi: B^G \to A^G$ between two shift-actions. We will say that this map is a {\em cellular map} if there is a finite subset $M$ of $G$ called a {\em memory set} and a function $\tau: B^M \to A$ such that $(\pi(y))(1_G) = \tau(\pr_M(y))$ for every $y \in B^G$.
In fact, by the famous Curtis-Hedlund-Lyndon theorem, a map from $A^G$ to $B^G$ is a cellular map iff it is equivariant and continuous (see e.g. \cite[Theorem 1.8.1]{CSC10}). 

Importance of cellular map for our presentation stems from a fact that any factor map can be approximated by a cellular map.
\begin{prop}\label{prop: cellular approximation}
Let $\nu$ be a Borel invariant probability measure on $B ^G$. Let $\pi: B^G \to A^G$ be an equivariant map. 
For any $\varepsilon>0$ there is a cellular map $\pi_\varepsilon: B^G \to A^G$ such that $\nu(\lbrace y \in B^G \,: (\pi(y))(1_G) \neq (\pi_\varepsilon(y))(1_G)\rbrace) < \varepsilon$. 
\end{prop}

\begin{proof}
Consider the map $\bar{\pi}: B^G \to A$ defined as $\bar{\pi}(y) = (\pi(y))(1_G)$. 
For each $a \in A$ denote $W_a$ the preimage of $a$ under the map $\bar{\pi}$. By the Borel regularity and since the clopen sets form a basis for the topology on $B^G$, we can find for each $a \in A$ such a clopen set $W_a''$ that $\nu(W_a \Delta W''_a) < \varepsilon/(100\lvert A \rvert)^2$. Let us now order elements of $A$ in some way: $A = \lbrace a_1, \ldots, a_n \rbrace$. Next we put $W'_{a_i} = W''_{a_i} \setminus \bigcup_{j < i} W''_{a_j}$, for $i < n$, and $W'_{a_n} = B^G \setminus \bigcup_{i<n}W'_{a_i} =  B^G \setminus \bigcup_{i<n}W''_{a_i}$ . It  is now clear that all $W'_{a_i}$ are clopen, pairwise disjoint, and their union is the whole space $B^G$. We also note that $\nu(W'_{a_i} \Delta W_{a_i}) < \varepsilon/(2 \lvert A\rvert)$ for each $i = 1, \ldots, n$. We now define a map $\bar{\pi}_{\varepsilon}$ by $\bar{\pi}_{\varepsilon}(y) = a_i$, if $y \in W'_{a_i}$, and induce a cellular map $\pi_{\varepsilon}$ from it.
\end{proof}

Let $X$ be a standard Borel space or standard probability space. A {\em partition} of $X$ is a finite or countable collection of measurable subsets of $X$ whose union is the whole set. 
For two partitions $\alpha$ and $\beta$ we denote $\alpha \vee \beta$ the partition ${\lbrace U \cap V, \;\text{where } U \in \alpha, V\in \beta, U \cap V \neq \varnothing\rbrace}$.

Consider an action of a countable group $G$ on a standard probability space $(X,\mu)$. For a partition $\alpha$ of $X$ and an element $g \in G$ we denote $\alpha^g$ the partition $\lbrace g^{-1}U ,\;\text{where } U\in \alpha \rbrace$. For a finite subset $F$ of $G$, we denote $\alpha^F$ the partition $\bigvee_{g \in F} \alpha^g$.
Partition $\alpha$ is said to be generating if the smallest $\sigma$-subalgebra containing $\alpha^g$ for each $g \in G$ is (modulo sets of mesure $0$) the algebra of all measurable subsets of $X$. The following condition is well-known to be equivalent: there is a subset $X' \subset X$ of full measure such that for every pair $x_1, x_2 \in X'$ with $x_1 \neq x_2$ there is an element $g \in G$ such that $gx_1$ and $gx_2$ belong to different pieces of $\alpha$ (see e.g. \cite[Lemma 2.1]{Se15}). 

A probability vector is a countable or finite sequence of non-negative numbers whose sum is $1$. Let $(p_i)$ be a probability vector. Its Shannon entropy is defined as
\[
-\sum_{i} p_i \log(p_i), 
\]
with the convention $0 \log 0 = 0$, and denoted by $H(p)$.
Any partition of a probability space defines a probability vector. For a partition $\alpha$ we denote $H(\alpha)$ the entropy of the corresponding probability vector.

Let a countable amenable group $G$ act on a standard probability space $(X,\mu)$, let $\alpha$ be a partition of finite Shannon entropy. Let $(F_i)$ be a F\o lner sequence for $G$.
We denote
\[
\h_{G}(\alpha,X,\mu) = \lim_{i \to \infty}\frac{H(\alpha^{F_i})}{\lvert F_i\rvert}.
\] 
It is well known that this limit does not depend on the choice of the F\o lner sequence (see \cite[Chapter 9.3, remarks before Definition 9.3]{KL16}). 
The Kolmogorov-Sinai entropy $\h_G(X,\mu)$ of the action is defined by the formula 
\[
\sup_{H(\alpha)<\infty}\h_G(\alpha,X,\mu),
\]
the supremum is taken along all the partitions with finite entropy. By the Kolmogorov-Sinai theorem, we have 
\[
\h_G(X,\mu) = \h_G(\alpha,X,\mu),
\]
for any finite generating partition $\alpha$, see \cite[Theorem 9.8]{KL16}.

Consider the shift-action of a countable group $G$ on $A^G$. Preimages of elements of $A$ under the map $x \mapsto x(e)$ from $A^G$ to $A$ constitute a partition. This partition is called the canonical alphabet generating partition.

Let a countable group $G$ act on a standard probability space $(X,\mu)$. Suppose $\bar{\psi}$ is a measurable map from $X$ to a finite set $A$. Let $\alpha'$ be the corresponding partition of $X$ (namely, the partition consisting of preimages under $\bar{\psi}$ of elements in $A$). Suppose $\alpha'$ is a generating partition. As we already noticed, the extended map $\pi: X \to A^G$ is a bijection on a subset of full measure. We endow now $A^G$ with the push-forward measure $\nu = \psi(\mu)$. It follows that $\psi$ is an isomorphism of $G$-actions. Let $\pi$ be an inverse of $\psi$ on a conull set, we can assume that this map is Borel ($\psi$ is a Borel map on a Borel set of full measure, an inverse of a Borel one-to-one map is a Borel map, see \cite[Corollary 15.2]{Ke95}). So $\pi$ is an isomorphism from $(A^G, \nu)$ to $(X,\mu)$. This map is also an isomorphism of actions. Alltogether we constructed a {\em symbolic representation}, corresponding to partition $\alpha'$, of the action of $G$ on $(X,\mu)$.

An action of a countable group on a standard probability space is called {\em essentially free} if it is free on the orbit of almost every point.

The following generating partition result was proved by Seward, \cite[Corollary 1.3]{Se19} (see also \cite[Corollary 2.9]{SeTD}):

\begin{theor}\label{thm: small entropy generator}
Consider an essentially free ergodic action of a countable amenable group $G$ on a standard probability space $(X,\mu)$. If $\h_G(X,\mu) < + \infty$, then for every $\varepsilon > 0$ there is a finite alphabet $A$ and a symbolic representation of this action by the shift-action on $(A^G, \nu)$ such that the canonical alphabet generating partition has entropy smaller than $\h_G(X,\mu)+ \varepsilon$.
\end{theor}


\subsection{Topological entropy}
For a compact metric space $(X,l)$ we denote $N_{\varepsilon}(X,l)$ the size of a mininmal $\varepsilon$-spanning net. 

Let $X$ be a compact metrizable space. Let $G$ be a countable amenable group acting by homeomorphisms on $X$.
Let us fix a compatible metric $l$ on $X$.  Suppose $S$ is a finite subset of $G$; we denote 
\[
l^S(x,y) = \sup_{g \in S} l(gx, gy),
\]
for $x,y \in X$. Let us fix a F\o lner sequence $(F_i)$ for $G$.
{\em Topological entropy} of the action is defined as
\[
\sup_{\varepsilon>0} \limsup_{i \to \infty} \frac{\log N_{\varepsilon}(X,l^{F_i})}{\lvert F_i\rvert}
\] and denoted by $\h^{top}_G(X)$. It is well known that this does not depend on the choice of the compatible metric and F\o lner sequence (see \cite[Theorem 9.39]{KL16}).

For a subshift $X \subset A^G$ there is the following equivalent definition of the topological entropy (see \cite[Example 9.4]{KL16}):

\[
\h^{top}_G(X) = \lim_{i \to \infty} \frac{\log \lvert \pr_{F_i}(X) \rvert}{\lvert F_i \rvert}.
\]

We will need the following implication of the variational principle.

\begin{theor}\label{thm: variational}
Let $A$ be a finite set and $G$ be an amenable group, let $X \subset A^G$ be a subshift. There is such an invariant Borel probability measure $\mu$ on $X$ that 
\[
\h_G(X,\mu) = \h^{top}_G(X).
\]
\end{theor} 
\begin{proof}
This follows from the variational principle together with the upper-semicontinuity of measure entropy as a function on the set of invariant measures for an expansive dynamical system. See Chapter 5 from \cite{Mou85}.
\end{proof}

\subsection{Asymptotic complexity}
Let $A$ be a finite alphabet. Consider the shift-action of a computable amenable group $G$ on $A^G$. Let $\F=(F_i)$ be a F\o lner sequence. 
As we discussed earlier, all partial maps from $G$ to $A$ with finite domains form naturally a constructible class.
The {upper asymptotic complexity} of a point $x \in A^G$ relative to $\F$ is defined as 
\[
\limsup_{i \to \infty} \frac{C(\pr_{F_i}(x))}{\lvert F_i\rvert}
\]
and denoted by $\uacf(x)$. 
The {lower asymptotic complexity} of a point $x \in A^G$ relative to $\F$ is defined as 
\[
\liminf_{i \to \infty} \frac{C(\pr_{F_i}(x))}{\lvert F_i\rvert}
\]
and denoted by $\lacf(x)$.

\begin{prop}\label{prop: asymtotic complexity equivalent def}
If $\F=(F_i)$ is a modest F\o lner sequence, then
\[
\uacf(x) = \limsup_{i \to \infty} \frac{C(\cont(\pr_{F_i}(x)))}{\lvert F_i\rvert}
\]
for every $x \in A^G$; analogous equality holds for the lower asymptotic complexity. 
\end{prop}

\begin{proof}
Note that $\cont: A^{\Subset G} \to A^*$ is a computable function, so
\[
\frac{C(\cont(\pr_{F_i}(x)))}{\lvert F_i\rvert} \leq \frac{C(\pr_{F_i}(x)) + O(1)}{\lvert F_i\rvert},
\]
which implies that
\[
\limsup_{i \to \infty} \frac{C(\cont(\pr_{F_i}(x)))}{\lvert F_i\rvert} \leq \limsup_{i \to \infty} \frac{C(\pr_{F_i}(x))}{\lvert F_i\rvert}.
\]
Now let $q: Fin(G) \times A^* \to A^{\Subset G}$ be a partial computable function such that $t = q(\dom t, \cont t)$ for each $t \in A^{\Subset G}$ (it constructs a partial map from a finite subset and a string of matching size).
It is easy to note that 

\begin{multline*}
\frac{C(\pr_{F_i}(x))}{\lvert F_i\rvert} = \frac{C(q(F_i,\cont(\pr_{F_i}(x))))}{\lvert F_i\rvert}\\ \leq \frac{C(F_i) + C(\cont(\pr_{F_i}(x))) + 2 \log C(F_i) + O(1)}{\lvert F_i\rvert}\\ = \frac{o(\lvert F_i\rvert) + C(\cont(\pr_{F_i}(x))) + 2 \log o(\lvert F_i\rvert) + O(1)}{\lvert F_i\rvert}.
\end{multline*}
Thus we have 
\[
\limsup_{i \to \infty} \frac{C(\pr_{F_i}(x))}{\lvert F_i\rvert} \leq\limsup_{i \to \infty} \frac{C(\cont(\pr_{F_i}(x)))}{\lvert F_i\rvert}.
\]

\end{proof}

\section{Reflections concerning modest F\o lner sequences}\label{sec: modest}

We remind that a F\o lner sequence $\F = (F_i)$ is called modest if $C(F_i) = o(\lvert F_i \rvert)$.

The following proposition will be needed to apply the Borel-Cantelli lemma in the proof of Proposition \ref{prop: measure lower bound}. 
\begin{prop}\label{prop: modest convergence}
Let $(F_i)$ be a modest F\o lner sequence in a computable amenable group $G$. If $F_i \neq F_j$ for  every pair $(i,j)$ with $i \neq j$, then for every $\varepsilon> 0$ the following serie converges:
\[
\sum_{i \in \N} {2^{-\varepsilon \lvert F_i \rvert}}.
\]
\end{prop}
\begin{proof}
Throwing away a finite initial segment of the F\o lner sequence, we may assume that $C(F_i) < \varepsilon/2 \cdot \lvert F_i \rvert$ for every $i$. We rearrange the sum:
\[
\sum_{k \in \N} 2^{- \varepsilon k}\lvert \lbrace i : \lvert F_i \rvert = k \rbrace\rvert.
\]
For each $F_i$ with $\lvert F_i\rvert = k$ we have $C(F_i)< \varepsilon k /2$; so for each $k$ we have $\lvert \lbrace i : \lvert F_i \rvert = k \rbrace\rvert \leq 2^{\varepsilon k /2}$. This means that the sum is bounded by the geometric series:
\[
\sum_{k \in \N} 2^{-\varepsilon k / 2}.
\]
\end{proof}

The next theorem provides a geometric criterion for a F\o lner sequence to be modest. We note that it is not used in the sequel, but shows that main results of this paper are applicable in a wide variety of situations.
\begin{theor}\label{thm: geometric criterion}
Suppose $G$ is a computable amenable group which is finitely generated. Let us fix a finite symmetric generating set $S$. This defines the Cayley graph structure on $G$. Let $(F_i)$ be a F\o lner sequence. If each $F_i$ is an edge-connected subset of the Cayley graph and contains the group identity, then $(F_i)$ is a modest F\o lner sequence.
\end{theor}

\begin{proof}
Algorithm \ref{dfs} (see Appendix \ref{app: algorithms}) takes a finite connected subset $T$ which contains $1_G$ and reversibly encodes it as a binary string consisting of $\lvert T \rvert$ $1$'s and $\lvert ST \setminus T \rvert$ $0$'s. That algorithm is based on the standard depth-first search algorithm for graphs (see \cite{CLRS}, Chapter 22.3).

I would like to note that the output string of Algorithm \ref{dfs} uniquelly determines the input set. Algorithm \ref{undfs} (Appendix \ref{app: algorithms}) decodes the set from the string. There $B(1_G,n)$ stands for the ball of radius $n$ in Cayley graph around the group identity.
If we apply the encoding to an element $F_i$ of the F\o lner sequence, we get the binary string of length $(1+ o_i(1))\lvert F_i\rvert$; this string contains a negligible amount of $0$'s; hence, by Lemma \ref{lem: shannon bound}, we have $C(F_i) = o(\lvert F_i\rvert)$.

\end{proof}

\section{Measure entropy; lower bound for complexity}\label{sec: measure lower bound}

Let $\alpha$ be a partition of a standard probability space $(X, \mu)$. For $x \in X$ denote $\alpha(x)$ the element of $\alpha$ that contains $x$. We will need the following generalization of the Shannon-McMillan-Breiman theorem due to Lindenstrauss \cite{L01}:

\begin{theor}
Consider an ergodic action of a countable amenable group $G$ on a standard probability space $(X, \mu)$. Let $\alpha$ be a partition of finite Shannon enropy. Let $(F_i)$ be a tempered F\o lner sequence. For $\mu$-a.e. $x \in X$ holds
\[
\lim_{i \to \infty} \frac{-\log \mu(\alpha^{F_i}(x))}{\lvert F_i\rvert} = \h_G(\alpha, X,\mu).
\] 
\end{theor}

\begin{prop}\label{prop: measure lower bound}
Let $\mu$ be an ergodic invariant measure for the shift-action of a computable amenable group $G$ on $A^G$. If $\F=(F_i)$ is a modest tempered F\o lner sequence, then for $\mu$-a.e. $x\in A^G$ we have
\[
\lacf(x) \geq \h_G(A^G,\mu).
\]
\end{prop}

\begin{proof}
Denote $h=\h_G(A^G,\mu)$.
Without loss of generality we may assume that the F\o lner sequence does not have repeating elements.
Let $\alpha$ be the canonical alphabet generating partition. Let $\varepsilon,\delta>0$. By the theorem above, there is a subset $R \subset A^G$ with $\mu(R)> 1 - \delta$ and a number $N$ such that for each $x \in R$ and $i>N$ we have 
\[
\log \mu(\alpha^{F_i}(x)) \leq -(h - \varepsilon) \lvert F_i\rvert.
\]
In order to prove the proposition it is enough to show that for $\mu$-a.e. $x \in R$ we have
\[
\lacf(x) \geq h - 2\varepsilon.
\]
For each $i>N$, denote $S_i$ the subset of all $x \in R$ such that $C(\pr_{F_i}(x)) < (h - 2\varepsilon) \lvert F_i \rvert$, these sets are measurable since they are unions of cylinder sets. By the Borel-Cantelli lemma, it suffices to show that
\[
\sum_{i > N} \mu(S_i) < +\infty.
\]
We proceed by showing that this is the case.
For each $i>N$ there are at most $2^{(h-2\varepsilon)\lvert F_i\rvert}$ such $t \in A^{F_i}$ that $C(t) < (h-2\varepsilon)\lvert F_i\rvert$. This implies that 
\[
\mu(S_i) \leq 2^{(h-2\varepsilon)\lvert F_i\rvert} \cdot 2^{-(h-\varepsilon)\lvert F_i\rvert} = 2^{-\varepsilon \lvert F_i\rvert}.
\]
So the desired convergence holds due to Proposition \ref{prop: modest convergence}.

\end{proof}

\section{Measure entropy; upper bound for complexity}\label{sec: measure upper bound}

Let $\mu$ be an ergodic measure on $A^G$; we will show that (under some additional requirements) the upper asymptotic complexity of $\mu$-a.e. point is bounded from above by the Kolmogorov-Sinai entropy. Here is the general outline of the proof. We first prove a ``trivial'' bound: for an ergodic invariant measure on a shift-action, for almost every point the asymptocic complexity is bounded from above by the Shannon entropy of the cannonical alphabet generating partition.
We then show that the asympotic complexity is essentially monotone relative to factor maps. This is done first for cellular maps (which is immediate); then we extend this to any factor-map by means of an approximation argument. Theorem \ref{thm: small entropy generator} of Seward and Tucker-Drob implies that the desired bound on the complexity holds if the action is essentially free. We relieve the freeness assumption by taking a product with a Bernoulli action of small entropy. 

We will use the following ergodic theorem due to Lindenstrauss(\cite{L01}:

\begin{theor}\label{thm: lin ergodic}
Let $G$ be a countable amenable group acting ergodically on a standard probability space $(X,\mu)$. Let $f$ be an $L^1$ function on $(X,\mu)$. If $(F_i)$ is a tempered F\o lner sequence, then
\[
\lim_{i \to \infty}\frac1{\lvert F_i \rvert}\sum_{g \in F_i} f(g x) = \expect_{(X,\mu)} f,
\]
for $\mu$-a.e. $x \in X$.
\end{theor}

Our first goal is to establish a ``trivial'' bound in Lemma \ref{lem: simple bound for complexity via entropy}.

Let $A$ be a finite alphabet. For a word $w \in A^n$ we denote $p(w)$ the probability vector for occurence rates of letters from $A$ in $w$.

The following is Lemma 146 from the book \cite{SUV}.

\begin{lem}\label{lem: shannon bound}
For every $w \in A^*$ holds
\[
C(w) \leq \lvert w\rvert(H(p(w)) + o_{\lvert w\rvert}(1)).
\]
\end{lem}

\begin{lem}\label{lem: simple bound for complexity via entropy}
Consider the shift-action of a computable amenable group $G$ on $A^G$ for a finite set $A$. Let $\mu$ be an ergodic invariant measure. Let $\alpha$ be the canonical alphabet generating partition. If $\F =(F_i)$ is a tempered F\o lner sequence, then
\[
\limsup_{i \to \infty} \frac{C(\cont( \pr_{F_i} (x)))}{\lvert F_i\rvert} \leq H(\alpha),
\]
for $\mu$-a.e. $x \in X$.
If in addition $\F$ is modest, then 
\[
\uacf(x)\leq H(\alpha),
\]
for $\mu$-a.e. $x \in X$.
\end{lem}
\begin{proof}
I claim that for $\mu$-a.e. $x \in A^G$, we have that $p(\cont(\pr_{F_i}(x)))$ tends to the probability vector of $\alpha$. Indeed, denote $W_a$ for $a \in A$ the set of all $x \in A^G$ with $x(1_G) = a$. The occurence rate $p(\pr_{F_i}(x))(a)$ of $a$ in $\pr_{F_i}$ is equal to 
\[
\frac{1}{\lvert F_i \rvert} \sum_{g \in F_i} I_{W_a}(gx),
\]
where $I_{W_a}$ is the indicator function of set $W_a$. The claim follows from the Lindenstrauss ergodic theorem applied to function $I_{W_a}$. 

The first assertion is now a consequence of the previous lemma. The second one follows easily now from Proposition \ref{prop: asymtotic complexity equivalent def}.
\end{proof}

Next, we would like to prove that asymptotic complexity does not increase under factor-maps. This is easy to establish for cellular maps. Then we observe that two colorings from $A^G$ which are mostly the same (in the sense of Hamming distance), have close asymptotic complexities. We then note that any equivariant map can be approximated by a cellular map. The desired monotonicity follows.

Let $w'$ and $w''$ be two words in $A^*$ of the same length. The {\em Hamming distance} between $w'$ and $w''$ is defined as 
\[
\frac{\lvert i: w'(i) \neq w''(i) \rvert}{\lvert w'\rvert},
\] 
and denoted by $\HD(w', w'')$.

The lemma below shows that two words have close complexity rates, given that these words are close in the Hamming metric.

\begin{lem}\label{lem: words hamming and complexity}
For two words $w'$,$w''$ in $A^*$ of the same length $n$ the following holds:
\[
\lvert C(w') - C(w'')\rvert \leq n \left(H(\HD(w',w''), 1-\HD(w',w'')) +  \HD(w',w'')\log \lvert A \rvert + o_n(1)\right).
\]
\end{lem}
\begin{proof}
In order to recover $w''$ from $w'$ it is enough to provide the word $u$ from $\lbrace 0,1 \rbrace^n$ that will encode the places where $w'$ and $w''$ differ, and the word $v$ from $A^{n \HD(w', w'')}$ that will encode the letters to substitude: it will have $1$'s in the places where $w'$ and $w''$ differ, and $0$' elsewhere.
The first word contains exactly $n \cdot \HD(w',w'')$ of $1$'s. So its complexity is bounded by 
\[
n \cdot H(\HD(w',w''), 1-\HD(w',w'')) +o(n), 
\]
by Lemma \ref{lem: simple bound for complexity via entropy}.
The second has complexity bounded by 
\[
n \HD(w', w'') \log \lvert A\rvert + O(1), 
\]
by Statement \ref{compl prop: trivial bound} from  Propostition \ref{prop: compl is factor-monotone}.

Let $q$ be a computable function that takes word $w'$, the difference-encoding string $u$, and the string of substitute letters $v$, and outputs $w''$.
We can see that 
\begin{multline*}
C(w'') = C(q(w',u,v)) \leq C(w',u,v) + O(1) \\ \leq C(w') + 2\log C(w') + C(u) + 2\log C(u) + C(v) \\= C(w') + C(u) + C(v) + o(n) \\\leq C(w') + n \cdot H(\HD(w',w''), 1-\HD(w',w'')) + n \cdot \HD(w',w'') + o_n(n). 
\end{multline*}  
We used monotonicity of the Kolmogorov complexity under computable maps and then the bound for the Kolmogorov complexity of a tuple, see Proposition \ref{prop: compl properties}.
We can obtain a similar upper bound for $C(w')$ in terms of $C(w'')$. 
\end{proof}

For two elements $t_1, t_2 \in A^{\Subset G}$ with $\supp(t_1) = \supp(t_2)$, we define the Hamming distance between them by 
\[
\frac{\vert \lbrace g \in supp(t_1),\; t_1(g)  \neq t_2(g)\rbrace\rvert}{\lvert\supp(t_1)\rvert},
\]
and denote it as $\HD(t_1, t_2)$. 

Let $\F =(F_i)$ be a F\o lner sequence. For two elements $x_1, x_2 \in A^G$, we define the {\em asymptotic Hamming distance} between them relative to $\F$ by the formula
\[
\limsup_{i \to \infty} \HD(\pr_{F_i}(x_1), \pr_{F_i}(x_2)),
\]
and denote it as $\dhf(x_1, x_2)$.

\begin{lem}
Let $A$ and $B$ be finite sets. Suppose $\pi'$ and $\pi''$ are two measurable equivariant maps from $B^G$ to $A^G$. Suppose $\mu$ is an ergodic invariant measure on $B^G$. Let $W$ be the set of all such $y \in B^G$ that $(\pi'(y))(1_G) \neq (\pi''(y))(1_G)$, where $1_G$ stands for the identity element in group $G$. Let $\F = (F_i)$ be a tempered F\o lner sequence. The for $\mu$-a.e. $y \in B^G$ holds
\[
\dhf(\pi'(y),\pi''(y)) = \mu(W).
\]

\end{lem}
\begin{proof}
It is easy to note that $(\pi'(y))(g) \neq (\pi''(y))(g)$ is equivalent to $(g\pi'(y))(1_G) \neq (g\pi''(y))(1_G)$, which in turn is equivalent to $(\pi'(gy))(1_G) \neq (\pi''(gy))(1_G)$. The latter is equivalent to $gy \in W$.
Now the statement follows from Lindenstrauss' ergodic theorem (Theorem \ref{thm: lin ergodic}) applied to the indicator function of set $W$.
\end{proof}

\begin{lem}\label{lem: map approximation and complexity}
In the setting of the previous lemma assume $\F= (F_i)$ is a modest F\o lner sequence. For $\mu$-a.e. $y \in B^G$ we have
\[
\lvert \uacf(\pi'(y)) - \uacf(\pi''(y)) \rvert \leq H(\mu(W),1 - \mu(W)) + \mu(W)\log \lvert A\rvert.
\]
\end{lem}
\begin{proof}
Since the F\o lner sequence is modest, by Proposition \ref{prop: asymtotic complexity equivalent def}, it is enough to prove that 
\begin{multline*}
\limsup_{i \to \infty} \frac{\big\lvert C(\cont(\pr_{F_i}(\pi'(y))))- C(\cont(\pr_{F_i}(\pi''(y))))\big\rvert}{\lvert F_i\rvert} \\ \leq  H(\mu(W),1 - \mu(W)) + \mu(W)\log \lvert A\rvert.
\end{multline*}

By Lemma \ref{lem: words hamming and complexity}, we have

\begin{multline*}
\big\lvert C(\cont(\pr_{F_i}(\pi'(y)))) - C(\cont(\pr_{F_i}(\pi''(y)))) \big\rvert \\ \leq \lvert F_i \rvert \cdot  H\Big(\HD\big(\pr_{F_i}(\pi'(y)), \pr_{F_i}(\pi''(y))\big),  1 - \HD\big(\pr_{F_i}(\pi'(y)), \pr_{F_i}(\pi''(y))\big)\Big) \\+ \HD\big(\pr_{F_i}(\pi'(y)), \pr_{F_i}(\pi''(y))\big) \log \lvert A\rvert + o(\lvert F_i\rvert).
\end{multline*}
Next, we divide by $\lvert F_i \rvert$ and take the upper limit. Using the previous lemma, we get the desired.

\end{proof}

\begin{lem}\label{lem: complexity monotone cellular}
If $\pi: B^G \to A^G$ is a cellular map, then for every $y \in B^G$ and every F\o lner sequence $\F = (F_i)$ holds
\[
\uacf(\pi(y)) \leq \uacf(y).
\] 
\end{lem}
\begin{proof}
Let $M$ be a memory set of cellular map $\pi$, we assume that $1_G$ belongs to it. 
Abusing notation a little, we extend the cellular map to $B^{\Subset G}$. Intuitively, we apply the local transformation rule from the definition of a cellular map in each place where we have enough information to do so.
Let us make it more precise. First, Let $\bar{\pi} : B^{G} \to A$ be the generator map for the cellular map $\pi$ (that is, $\bar{\pi}(y)= (\pi(y))(1_G)$). We may extend  map $\bar{\pi}$ to the set of such $t \in B^{\Subset G}$ that $\supp(t) \supset M$ (since for any $y \in B^G$ the value $\bar{\pi}(y)$ is completely determined by $\pr_{M}(y)$). Now, for $t \in A^{\Subset G}$, we define $\pi(t)$ in such a way that $\supp(\pi(t))$ is the maximal set $T$ satisfying $MT \subset \supp (t)$, and $(\pi(t))(g) = \bar{\pi}(g t)$, for each $g \in T$ (we extended the shift action of the group to the set $A^{\Subset G}$ in a natural fashion). It is easy to see that the extended map $\pi$ we constructed is consistent with the original one in the following sense. If $T$ is a finite subset of $G$, then $\pr_{T}(\pi(y)) = \pi(\pr_{MT}(y)))$ for every $y \in B^G$. We note that the restriction of $\pi$ to finite configurations is a computable function from constructible class $B^{\Subset G}$ to constructible class $A^{\Subset G}$. Let $Q : B^{\Subset G} \times B^* \to A^{\Subset G}$ be the following partial computable function. It takes $t \in B^{\Subset G}$, and a word $w \in B^*$ of the length $\lvert M\supp(t) \setminus \supp(t) \rvert$ (if length of $w$ does no satisfy this requirement, the behaviour of $Q$ is undefined), and outputs $\pi(t')$, where $t'$ is the unique configuration from $B^{\Subset G}$ such that
\begin{itemize}
\item $\supp(t') = M \supp(t)$, 
\item $\pr_{\supp(t)}(t') = t$,
\item $\cont(\pr_{M \supp(t) \setminus \supp(t)}(t')) = w$.
\end{itemize}
Note that
\[
\pr_{F_i}(\pi(y)) = \pi(Q(\pr_{F_i}(y), \cont(\pr_{M \F_i \setminus F_i}(y)))),
\]
for each  $y \in B^{G}$. This implies that 
\begin{multline*}
C(\pr_{F_i}(\pi(y))) \leq  C(\pr_{F_i}(y)) + C(\cont(\pr_{MF_i \setminus F_i}(y))) \\+ 2 \log C(\cont(\pr_{MF_i \setminus F_i}(y))),
\end{multline*}
by Properties \ref{compl prop: monot} and \ref{compl prop: tuple} from Proposition \ref{prop: compl properties}.
Since $\lvert MF_i \setminus F_i\rvert = o(\lvert F_i\rvert)$ (by the definition of F\o lner sequence) and 
\[
C(\cont(\pr_{MF_i \setminus F_i}(y))) \leq \log \lvert B \rvert \cdot \lvert MF_i \setminus F_i\rvert + O(1)
\] 
(by Property \ref{compl prop: trivial bound} from Proposition \ref{prop: compl properties}), we have 
\[
C(\cont(\pr_{MF_i \setminus F_i}(y))) = o(\lvert F_i \rvert), 
\]
and 
\[
\log C(\cont(\pr_{MF_i \setminus F_i}(y))) = o(\lvert F_i \rvert). 
\]
Thus we have 
\[
\uacf(\pi(y)) \leq \uacf(y).
\]

\end{proof}

We are now able to prove the main monotonicity statement for asymptotic Kolmogorov complexity:

\begin{prop}\label{prop: compl is factor-monotone}
Let $\pi: B^G \to A^G$ be a measurable equivariant map. Let $\F = (F_i)$ be a modest and tempered F\o lner sequence. Let $\mu$ be an ergodic invariant measure on $B^G$.
For $\mu$-a.e. $y\in B^G$ holds the following bound:
\[
\uacf(\pi(y)) \leq \uacf(y).
\] 
\end{prop}

\begin{proof}
Take arbitrary $\varepsilon>0$. We aim to prove that $\uacf(\pi(y)) \leq \uacf(y) + \varepsilon$ for $\mu$-a.e. $y \in B^G$. 
There is $\varepsilon'>0$ such that for every non-negative $\varepsilon''<\varepsilon'$ we have 
\[
H(\varepsilon'', 1-\varepsilon'') + \varepsilon'' \log \lvert A\rvert  \leq \varepsilon.
\]
Let $\pi'$ be a cellular map from $B^G$ to $A^G$ such that for the set $W$ of all the points $y \in B^G$ for which $(\pi(y))(1_G) \neq (\pi'(y))(1_G)$, we have $\mu(W) < \varepsilon'$ (see Proposition \ref{prop: cellular approximation}).    Lemma \ref{lem: complexity monotone cellular} implies that $\uacf(\pi'(y)) \leq \uacf(\pi'(y))$ for every $y \in B^G$. Lemma \ref{lem: map approximation and complexity} implies that $\uacf(\pi(y)) \leq \uacf(\pi'(y)) + \varepsilon$ for $\mu$-a.e. $y \in B^G$. It follows that $\uacf(\pi(y)) \leq \uacf(y) + \varepsilon$ for $\mu$-a.e. $y \in B^G$.
\end{proof}

The next proposition is the main result of this section.

\begin{prop}\label{prop: measure upper bound}
Let $A$ be a finite set. Consider the shift-action of a computable amenable group $G$ on $A^G$. Let $\mu$ be an invariant ergodic measure on $A^G$. Let $\F=(F_i)$ be a modest and tempered F\o lner sequence. We have
\[
\uacf(x) \leq \h_G(A^G,\mu),
\]
for $\mu$-a.e. $x \in A^G$.
\end{prop}
\begin{proof}
Denote $h = \h_G(A^G,\mu)$. Fix any $\varepsilon>0$. We will prove that 
\[
\uacf(x) \leq h + \varepsilon,
\]
for $\mu$-a.e. $x \in A^G$.
Let $G \curvearrowright (A'^G, \mu')$ be a Bernoulli action of entropy smaller than $\varepsilon/2$. Consider the product-action $G \curvearrowright ((A \times A')^G,\mu \otimes \mu')$. This product is essentially free. It is also ergodic. Indeed, the product of a Bernoulli action with itself is ergodic, so its product with any other ergodic action is also ergodic (see Theorem 3.11 from the book \cite{Gl03}). Entropy of the product is smaller than $h+\varepsilon/2$ (\cite[Theorem 9.16]{KL16}). By Theorem \ref{thm: small entropy generator}, it has a symbolic representation $G \curvearrowright (B^G,\nu)$ such that for the canonical alphabet generating partition $\beta$ holds $H(\beta) \leq h + \varepsilon$. 
So we have a factor-map $\pi$ from $(B^G,\nu)$ to $(A^G,\mu)$.
By Lemma \ref{lem: simple bound for complexity via entropy}, we have $\uacf(y) \leq h + \varepsilon $ for $\nu$-a.e. $y \in B^G$.
By Proposition \ref{prop: compl is factor-monotone}, we have that for $\nu$-a.e. holds $\uacf(\pi(y))\leq \uacf(y)$. This implies that $\uacf(x) \leq h + \varepsilon$ for $\mu$-a.e. $x \in A^G$.
\end{proof}

\section{Topological entropy; lower and upper bounds for complexity}\label{sec: topological lower bound}

In this section we first show that for a subshift $X \subset A^G$ there are points with lower asymptotic complexities not smaller than the topological entropy of the subshift, and that there is a continuum of such points if $X$ has continuum cardinality.

Afterwards, we will prove that the topological entropy of a subshift bounds from above asymptotic complexities of its points. The idea behind the proof is the following. We remind that, if $y$ is an element of a finite set $W$, then we have the bound 
\[
C(y) \leq C(W) + \log \lvert W \rvert + 2 \log (C(W)) + O(1).
\]
We use a computable version of the Ornstein-Weiss covering lemma to show that for every $\varepsilon>0$ there is a partial computable function $Q: \Fin(G) \to A^{\Subset G}$ and a number $N$ such that $\pr_{F_i}(X) \subset Q(F_i)$, and $\log \lvert Q(F_i)\rvert \leq \lvert F_i \rvert (\h^{top}_G(X) + \varepsilon)$, for every $i>N$. This together with the inequality above would imply the desired bound.

\begin{prop}\label{prop: topological entropy lower bound}
Let $A$ be a finite alphabet and let $G$ be a computable amenable group. Let $X \subset A^G$ be a subshift. If $\F$ is a modest F\o lner sequence in $G$, then there is a point $x \in A^G$ such that 
\[
\lacf(x) \geq \h^{top}_G(X).
\]
If $X$ has continuum cardinality, then there is a continuum of such points.
\end{prop}
\begin{proof}
By Theorem \ref{thm: variational}, there is an ergodic invariant measure $\mu$ such that $\h^{top}_G(X) = \h_G(X,\mu)$. There is a point $x \in X$ with 
\[
\lacf(x) \geq \h_G(X,\mu)
\]
by Proposition \ref{prop: measure lower bound}.
If $\h^{top}_G(X) = 0$, then all the points from $X$ satisfy the requirement. If $\h^{top}_G(X) > 0$, then measure $\mu$ is non-atomic(since otherwise the coresponding maeure entropy would be zero), and hence it is continuous, This means that any subset of full measure has the continuum cardinality.
\end{proof}


The following is a suitable for our purpose variant of the Ornstein-Weiss lemma. It deals with the problem of covering the elements of a F\o lner sequence by translates of other F\o lner sets. While it is impossible to achieve a perfect coverage without intersections, one can have an almost cover with only small intersections. 
\begin{lem}
Let $G$ be a computable amenable group. Let $(F_i)$ be a F\o lner sequence. For every rational $\varepsilon>0$ there is a finite sequence $j_1, \ldots, j_k$, a number $N$, and program $R$ such that the following holds. Program $R$ receives a finite subset $T$. If $T = F_i$ for some $i>N$, then it outputs the sequence of subsets $R_1=R_1(T), \ldots, R_k=R_k(T)$ satisfying the assertions below; otherwise its behaviour is undefined.
\begin{enumerate}
\item $\bigcup_{s=1}^k F_{j_s} R_s \subset T$;
\item $\lvert T \setminus \bigcup_{s=1}^k F_{j_s} R_s\rvert \leq \varepsilon \lvert T \rvert$;
\item $\sum_{s=1}^k \lvert F_{j_s} \rvert \cdot \lvert R_s \rvert \leq (1+\varepsilon) \left\lvert \bigcup_{s=1}^k F_{j_s} R_s \right\rvert$;
\item $\sum_{s=1}^k \lvert F_{j_s} \rvert \cdot \lvert R_s \rvert \leq (1+\varepsilon) \lvert T\rvert$.
\end{enumerate} 
\end{lem}
\begin{proof}
If we drop the requirement for $R_1, \ldots, R_k$ to be generated by a program, the statement becomes a variant of the Ornstein-Weiss covering lemma (see \cite{OW87} and \cite{KL16} Lemma  9.23). The program will just search for the output among finitely many variants.
\end{proof}

\begin{prop}\label{prop: topological upper bound}
Let $G$ be a computable amenable group. Let $A$ be a finite alphabet. Suppose $X \subset A^G$ is a subshift. Let $(F_i)$ be a modest F\o lner sequence for group $G$. The bound
\[
\uacf(x) \leq \h^{top}_G(X)
\]
hold for every $x \in X$.
\end{prop}
\begin{proof}
Denote $h=\h^{top}_G(X)$.
Take an arbitrary rational $\varepsilon$. It suffices to prove that
\[
\uacf(x) \leq (1+\varepsilon)(h +\varepsilon)  + \varepsilon \log\lvert A\rvert \label{eq: upper bound}\tag{*},
\]
for every $x \in X$.
Throwing away a finite initial part of $(F_i)$, we may assume that 
\[
\log\left\lvert \pr_{F_i}(X)\right\rvert < (h+ \varepsilon) \lvert F_i \rvert
\]
for every $i$.
We apply the previous lemma. This gives us a collection $F_{j_1}, \ldots, F_{j_k}$, a number $N$, and a program $R$. We record the finite subsets $\pr_{F_{j_1}}(X), \ldots,\pr_{F_{j_k}}(X)$ for the future use.
Now we will consruct a program $Q$ that will take a finite subset $T$ of $G$, and if $T=F_i$ for some $i>N$, it will output the set $Q(T)$ of all $y \in A^T$ satisfying $\pr_{F_{j_s}m}(y) \in \pr_{F_{j_s}m}(X)$ for every $s \in 1..k$ and $m \in R_s(T)$. Note  that $\pr_{F_{j_s}m}(y) \in \pr_{F_{j_s}m}(X)$ is equivalent to $\pr_{F_{j_s}}(my) \in \pr_{F_{j_s}}(X)$), so we only need the finite subsets $\pr_{F_{j_1}}(X), \ldots,\pr_{F_{j_k}}(X)$ for this. Alltogether, $Q$ will first apply program $R$ to $T$, and output (if $R$ finished correctly) exactly those $y \in A^T$ such that $\pr_{F_{j_s}}(my) \in \pr_{F_{j_s}}(X)$ for every $s \in 1..k$ and $m \in R_s(T)$.

I claim that for $i>N$ we have 
\[
\log\lvert Q(F_i) \rvert \leq \left((1+\varepsilon)(h +\varepsilon)  + \varepsilon \log \lvert A\rvert\right) \lvert F_i \rvert. \tag{**}\label{eq: upper bond Q}
\]
Indeed, denote 
\[
T'= \bigcup_{s=1}^k F_{j_s} R_s(F_i).
\]
It follows from the previous lemma that 
\[
\log\left\lvert \pr_{T'}(Q(F_i)) \right\rvert \leq \sum_{s=1}^k\sum_{m \in R_i(F_s)} \log \left\lvert \pr_{F_{j_s}m}(Q(F_i))\right\rvert  \leq (1+\varepsilon)(h+\varepsilon)\lvert F_i\rvert.
\]
Bound \eqref{eq: upper bond Q} follows since 
\[
\log\lvert Q(F_i)\rvert \leq \log\lvert \pr_{T'}Q(F_i)\rvert + \lvert F_i \setminus T'\rvert \log \lvert A \rvert \leq \log\lvert \pr_{T'}Q(F_i)\rvert + \varepsilon\lvert F_i \rvert \log \lvert A \rvert.
\]
Note that for $i>N$ we have $\pr_{F_i}(X) \subset Q(F_i)$. We also have 
\[
C(Q(F_i)) \leq C(F_i) + O_i(1), 
\]
by Property \ref{compl prop: monot} from Proposition \ref{prop: compl properties}. Property \ref{compl prop: small set compl} from Propostion \ref{prop: compl properties} implies that for every $y \in Q(F_i)$ and for $i>N$ we have
\[
C(y) \leq \left((1+\varepsilon)(h +\varepsilon)  + \varepsilon \log\lvert A\rvert \right)\lvert F_i \rvert + o(\lvert F_i \rvert).
\]
Hence, \eqref{eq: upper bound} holds for every $x \in X$.
\end{proof}

\appendix
\section{Encoding and decoding algorithms for Propostion \ref{thm: geometric criterion}}\label{app: algorithms}

These algorithms are based on the standard depth-first search algorithm for graphs (see \cite{CLRS}, Chapter 22.3). 

\begin{algorithm}\label{dfs}
\caption{Encoding of a connected set}
\SetKwFunction{encode}{encode}
\SetKwFunction{dfsvisit}{dfs\_visit}
\SetKwFunction{print}{print}
\KwData{A finite subset $T$ of the group}
\KwResult{an encoding binary string}
\encode \Begin{
	\For{ $g \in ST \cup T  $}{
		visited[$g$] $\leftarrow$ false\;
	}
	\dfsvisit($1_G$)
}
\dfsvisit ($h$) \Begin{
	\If{ not  visited[$h$] }{
		visited[$h$] $\leftarrow$ true\;
		\eIf{$h \in T$}{\label{dfs:keyline}
			\print{'1'}\;
			\For{$g \in S$}{
				\dfsvisit($gh$)\;
			}
		}{
			\print{'0'}\;
		}
	}
} 
\end{algorithm}

\begin{algorithm}\label{undfs}
\caption{Decoding of a set}
\KwData{encoding binary string}
\KwResult{the list of set elements}
\SetKwFunction{decode}{decode}
\SetKwFunction{undfsvisit}{undfs\_visit}
\SetKwFunction{print}{print}
\SetKwFunction{Kwread}{read}
\decode\Begin{
	n $\leftarrow$ length(input string)\;
	\For{ $g \in B(1_G,n)$}{
		visited[$g$] $\leftarrow$ false\;
	}
	\undfsvisit ($1_G$)\;
}
\undfsvisit ($h$) \Begin{
	\If{not visited $[h]$}{
		visited[$h$] $\leftarrow$ true\;
		\Kwread(bit)\;
		\If{bit}{
			\print($g$)\;
			\For{$g \in S$}{
				\undfsvisit ($gh$)\;
			}
		}
	}
} 
\end{algorithm}

\end{document}